\documentclass[11pt]{article}
\usepackage{mathrsfs}
\usepackage{amscd}
\usepackage{amsmath,amsfonts}
\usepackage{amssymb,amscd,amsthm}
\usepackage{indentfirst,graphics,epsfig,psfrag}
\usepackage{ifpdf}
\usepackage{enumerate}
\usepackage{appendix}
\usepackage{enumerate}
\usepackage{color}
\usepackage{lineno}
\newtheorem{theorem}{Theorem}[section]
\newtheorem{lemma}[theorem]{Lemma}

\newtheorem{problem}[theorem]{Problem}

\title{\textbf{Tight toughness, isolated toughness and binding number bounds
		for the $\{K_2,C_n\}$-factors} }
\author{Xiaxia Guan, Tianlong Ma, Chao Shi\\
	\small School of Mathematical Sciences\\[-0.8ex]
	\small Xiamen University\\[-0.8ex]
	\small P. R. China\\
	\small\tt Email: gxx0544@126.com, tianlongma@aliyun.com, cshi@aliyun.com}

\date{}

\begin{document}

\maketitle
\begin{abstract}  The $\{K_2,C_n\}$-factor of a graph is a spanning subgraph whose each component is either $K_2$ or $C_n$. In this paper, a sufficient condition with regard to tight toughness, isolated toughness and binding number bounds to guarantee the existence of the $\{K_2,C_{2i+1}| i\geq 2 \}$-factor for any graph is obtained, which answers a problem due to Gao and Wang (J. Oper. Res. Soc. China (2021), https://doi.org/10.1007/s40305-021-00357-6).
\\
{\bf Keywords:}  Component factor; toughness; isolated toughness; binding number \\[2mm]
{\bf AMS subject classification 2020:} 05C42; 05C38.
\end{abstract}

\section{Introduction}

All graphs considered in this paper are simple, finite and undirected. Given a graph $G$, let $V(G)$ and $E(G)$ denote the sets of vertices and edges of $G$, respectively. For a subset $S\subseteq V(G)$, we use $G-S$ to
denote the subgraph of $G$ induced by the set $V(G)\setminus S$. Let
$N_G(S)$ denote the set of all neighbours of the vertices in $S$, and let $c(G-S)$ and $iso(G-S)$ denote the numbers of components and isolated vertices in $G-S$, respectively.

For a set of connected
graphs $\mathcal{S}$, a spanning subgraph $F$ of $G$ is called an \emph{$\mathcal{S}$-factor} if each component of $F$ is isomorphic to a member of $\mathcal{S}$. In particular an $\mathcal{S}$-factor is a $P_{\geq k}$-factor if $\mathcal{S}=\{P_k, P_{k+1},\ldots\}$, where $P_k$ is a path with $k$ vertices. We refer the readers to book \cite{Yu} for more details on graph factors.

In \cite{Chvatal}, Chv\'{a}tal defined the \emph{toughness} of a graph $G$, denoted by $t(G)$, as follows: if $G$ is a completed graph, then $t(G)=+\infty$; otherwise
\[t(G)=\min \left \{\frac{|S|}{c(G-S)}\bigg|S\subseteq V(G), c(G-S)\geq 2 \right\}.\]

Zhou et al. \cite{Zhou1} showed that for a graph $G$ with $t(G)\geq \frac{2}{3}$, $G$ has a $P_{\geq 3}$-factor.
Liu and Zhang \cite{Liu} showed that for a graph $G$ with $t(G)\geq 1$, $G$ has a $\{K_2,C_{2i+1}|i \geq 1\}$-factor.
The further relationships between toughness and graph factors were studied in \cite{Gao1, Yuan}.

To study fractional factors, a new parameter of a graph $G$, called \emph{isolated toughness} and denoted by $I(G)$, is introduced by Yang et al. in \cite{Yang}, which is defined as follows: if $G$ is a complete
graph, then $I(G)=+ \infty $; otherwise,
\[I(G)=\min \left \{\frac{|S|}{iso(G-S)}\bigg |S\subseteq V(G), iso(G-S)\geq 2 \right\}.\]

In \cite{Kano}, Kano et al. proved that if $I(G)\geq \frac{3}{2}$ for any graph $G$, then $G$ admits a $P_{\geq 3}$-factor. The fractional Tutte $1$-factor theorem \cite{Scheinerman} stated that a graph $G$ has a $\{K_2,C_{2i+1}|i \geq 1\}$-factor if and only if $iso(G-S)\leq |S|$ for any $S\subseteq V(G)$. It follows that $G$ admits a $\{K_2,C_{2i+1}|i \geq 1\}$-factor if and only if $I(G)\geq 1$.

Further, a variant of isolated toughness was defined by Zhang and Liu \cite{Zhang}, denoted by $I'(G)$: if $G$ is a completed graph, then $I'(G)=+\infty$; otherwise
\[I'(G)=\min \left \{\frac{|S|}{iso(G-S)-1}\bigg|S\subseteq V(G), iso(G-S)\geq 2 \right\}.\]

Recently, Gao and Wang \cite{Gao} proved that if $I'(G)> 2$ for any graph $G$, then $G$ admits a $P_{\geq 3}$-factor. Similar to isolated toughness, by the fractional Tutte $1$-factor theorem, we have that $G$ admits a $\{K_2,C_{2i+1}|i \geq 1\}$-factor if and only if $I'(G)> 1$.

In \cite{Woodall}, Woodall defined the \emph{binding number} of $G$, denoted by $bind(G)$, as follows:
\[bind(G)=\min \left \{\frac{|N_G(S)|}{|S|}\bigg|\emptyset\neq S\subseteq V(G), N_G(S)\neq V(G)\right\}.\]

Binding number plays an important role in the research of graph factors. In 1971, Anderson \cite{Anderson} showed for a graph $G$ of even order with $bind(G)\geq \frac{4}{3}$, $G$ has a $1$-factor. Later, Woodall \cite{Woodall} showed for a graph $G$ with $bind(G)\geq \frac{3}{2}$, $G$ has a $2$-factor. Gao and Wang \cite{Gao} proved that if $bind(G)\geq \frac{5}{4}$ for any graph $G$, then $G$ admits a $P_{\geq 3}$-factor. Moreover, they also proved that a graph $G$ admits a $\{K_2,C_{2i+1}|i \geq 1\}$-factor if and only if $bind(G)\geq 1$. We refer to the papers \cite{Katerinis, Zhou} for more details on binding number.

In \cite{Woodall}, Woodall proved the relation $t(G)\geq bind(G)-1$ for any graph $G$. For any graph $G$, let $S$ be a subset of vertices in $G$ such that $I(G)=\frac{|S|}{iso(G-S)}$. Note that $iso(G-S)\leq c(G-S)$, and thus
 \[t(G)\leq \frac{|S|}{c(G-S)}\leq \frac{|S|}{iso(G-S)}=I(G).\]
Let $Y$ be the set of the isolated vertices in $G-S$. Clearly, $iso(G-S)=|Y|$ and $N_G(Y)\subseteq S$. Then $|N_G(Y)|\leq |S|$. Therefore we have \[bind(G)\leq \frac{|N_G(Y)|}{|Y|}\leq \frac{|S|}{iso(G-S)}=I(G).\]
The relation $I(G)< I'(G)$ is trivial from the definitions. Thus we have $t(G)\leq I(G)< I'(G)$ and $bind(G)\leq I(G)< I'(G)$.
We think that these relations are simple and immediate, but it is a surprise that we can't find it in the existing literature.

In \cite{Gao}, Gao and Wang also simply surveyed the important results on component factors, and proposed the following problem.

 \begin{problem}\emph{\cite{Gao}}\label{Problem}
	What is the tight toughness, isolated toughness and binding number bounds
	for a graph $G$ to have a $\{K_2,C_{2i+1}|i \geq 2\}$-factor?
\end{problem}

In this paper, we give a sufficient condition with regard to tight toughness, isolated toughness and binding number bounds to guarantee the existence of the $\{K_2,C_{2i+1}|i \geq 2\}$-factor for any graph, and then Problem \ref{Problem} is answered.
\begin{theorem}\label{Mainth}
	If a graph $G$ of order at least $5$ satisfies one of the following four conditions:
	
	{\bf(i)} $t(G)\geq 1$;
	
	{\bf(ii)} $I(G)\geq 3$;
	
	{\bf(iii)} $I'(G)> 5$;
	
	{\bf(iv)} $bind(G)>\frac{4}{3}$,\\
	then $G$ admits a $\{K_2,C_{2i+1}|i\geq 2\}$-factor.
\end{theorem}

These bounds of toughness, isolated toughness and binding number our obtain are sufficient to admit a $\{K_2,C_{2i+1}|i\geq 2\}$-factor in any graph, but not necessary, which can be easily checked by a path of even order. Note that for a graph $G$, if $t(G)\geq 1$, our conclusion can be viewed as a generalization of the result that $G$ has a $\{K_2,C_{2i+1}|i \geq 1\}$-factor due to Liu and Zhang \cite{Liu}.

In particular, for the variant of isolated toughness, if we abandon several graphs of small order, we have the following further result.
\begin{theorem}\label{th}
	Let $G$ be a graph of order at least $16$ with $I'(G)> \frac{7}{2}$. Then $G$ admits a $\{K_2,C_{2i+1}|i\geq 2\}$-factor.
\end{theorem}

In next section, we shall complete the proofs of Theorems \ref{Mainth} and \ref{th}. Futhermore,
some classes of graphs are constructed to demonstrate the tightness for the bounds our obtain on the four parameters.

%\section{Preliminaries}
\section{Proof of Theorem \ref{Mainth}}
In this section, we divide Theorem \ref{Mainth} into four lemmas to  prove and demonstrate the tightness for the bounds our obtain.
We begin with a characterization on $\{K_2,C_{2i+1}| i\geq 2 \}$-factors, which plays a vital role in our forthcoming arguments.

We call a connected graph \emph{triangular-cactus} if its each block is a triangle. Clearly, $K_3$ is a triangular-cactus. We usually also view $K_1$ as a triangular-cactus. Let
$c_{tc}(G)$ denote the number of triangular-cacti components of $G$. In the following, we present the characterization on $\{K_2,C_{2i+1}| i\geq 2 \}$-factors due to
Cornuéjols and Pulleyblank \cite{Cornuejols}.

\begin{theorem}\emph{\cite{Cornuejols}} \label{factor}
	A graph $G$ has a $\{K_2,C_{2i+1}|i\geq 2\}$-factor if and only if $c_{tc}(G-S)\leq |S|$ for any $S\subseteq V(G)$.
\end{theorem}

We now prove our main results in the following.

\begin{lemma}
	Let $G$ be a graph of order at least $5$. If $G$ admits no $\{K_2,C_{2i+1}|i\geq 2\}$-factors, then $t(G)< 1$.
\end{lemma}

\begin{proof}
	By Theorem \ref{factor},
	there exists an $S\subseteq V(G)$ such that $c_{tc}(G-S)\geq |S|+1$. Note that $c (G-S)\geq c_{tc}(G-S)\geq |S|+1$. If $c(G-S)\geq 2$, then
	\[t(G)\leq \frac{|S|}{c (G-S)}\leq \frac{c(G-S)-1}{c(G-S)}<1.\]
   If $c(G-S)=1$, then $G$ is a triangular-cactus with at least $5$ vertices. Clearly, there exists a cut vertex $v$ in $G$, that is, $c (G-\{v\})\geq 2$. Thus
	\[t(G)\leq \frac{|\{v\}|}{c (G-\{v\})}\leq \frac{1}{2}.\]
	The proof is complete.
\end{proof}

{\bf Remark 1. }
Let us show that $t(G)<1$ is tight. Let $G_m$ be a graph obtained from a completed graph $K$ of order $m-1$ by joining $m$ independent vertices to each  vertex of $K$ for $m\geq 3$.
Note that $G_m$ has no $\{K_2,C_{2i+1}|i \geq 2\}$-factor as $c_{tc}(G_m-V(K))>|V(K)|$. If $c(G_m-S)\geq 2$, then $V(K)\subseteq S$ and $c(G_m-S)= 2m-1-|S|$, which implies that $|S|\geq m-1$. Hence
\[\frac{|S|}{c(G_m-S)}=\frac{|S|}{2m-1-|S|}=-1+ \frac{2m-1}{2m-1-|S|}. \]
Therefore, $t(G_m)=\frac{m-1}{m}<1$, and $\lim\limits_{m\to+\infty}t(G_m)=1$.

\begin{lemma}
	Let $G$ be a graph without $\{K_2,C_{2i+1}|i\geq 2\}$-factors. Then $I(G)<3$.
\end{lemma}
\begin{proof}
By Theorem \ref{factor}, there exists an $S\subseteq V(G)$ such that $c_{tc}(G-S)\geq |S|+1$. Let $a$ denote the number of $K_1$ in $G-S$. Let $T_1, T_2, \ldots,T_b$ be the components with at least $3$ vertices in $G-S$. So we rewrite
$c_{tc}(G-S)=a+b\geq |S|+1$, and then $|S|\leq a+b-1$.
Take two neighbours of a vertex of degree $2$ in $T_i$ to form a set $X_i$, and
set $X=\bigcup\limits_{i=1}^b X_i$. Clearly, $|X|=\sum\limits_{i=1}^b |X_i|=2b$ and $\sum\limits_{i=1}^{b}iso(T_i-X_i)=b$.
Set $Y=S\cup X$. Then we have
\[|Y|=|S|+ |X|\leq a+b-1+2b=a+3b-1,\]
and
\[iso(G-Y)=a+b.\]
Therefore,
\[I(G)\leq \frac{|Y|}{iso(G-Y)}\leq \frac{a+3b-1}{a+b}=3+\frac{-2a-1}{a+b}<3.\]
This completes the proof.
\end{proof}

{\bf Remark 2. }
Let us show that $I(G)<3$ is tight. Let $H_m$ be a graph obtained from completed graph $K$ of order $m-1$ by joining each vertex of $m$ disjoint $K_3$'s to each vertex of $K$ for $m\geq 3$.
Note that $H_m$ has no $\{K_2,C_{2i+1}|i \geq 2\}$-factor as $c_{tc}(H_m-V(K))>|V(K)|$. Let
$S\subseteq V(H_m)$ such that $iso(H_m-S)\geq 2$.
Then $V(K)\subseteq S$. So we have $iso(H_m-S) \leq m$ and $|S|\geq m-1+2iso(H_m-S)$.
Hence
\[\frac{|S|}{iso(H_m-S)}\geq\frac{m-1+2iso(H_m-S)}{iso(H_m-S)}\geq \frac{3m-1}{m}.\]
In particular, if we take two vertices from each $K_3$ and $V(K)$ as the set $S'$, then
$\frac{|S'|}{iso(H_m-S')}=\frac{3m-1}{m}$.
Therefore $I(H_m)=\frac{3m-1}{m}<3$,
and $\lim\limits_{m\to+\infty}I(H_m)=3$.

\begin{figure}[!t]
	\begin{center}
		\includegraphics[width=0.90\textwidth,height=20mm]{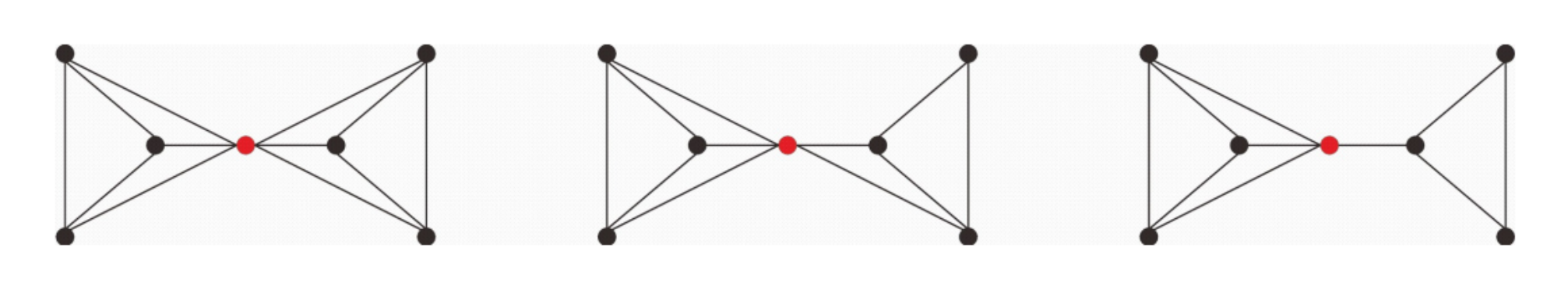}
		\put(-257,-29){\mbox{\scriptsize\quad Fig. 1. Three exceptions in Lemma \ref{I'}.}}
	\end{center}
\end{figure}
\begin{lemma}\label{I'}
	Let $G$ be a graph of order $n$ containing no $\{K_2,C_{2i+1}|i\geq 2\}$-factors with $n\geq5$. Then
	
	{\bf(i)} $I'(G)\leq 5$;
	
	{\bf(ii)} $I'(G)\leq 4$ if $G$ is not one of three graphs in Fig 1;
	
	{\bf(iii)} $I'(G)\leq \frac{11}{3}$ if $n\notin\{6,7,11\}$;
	
	{\bf(iv)} $I'(G)\leq \frac{7}{2}$ if $n\notin\{6,7,11,15\}$.
\end{lemma}
\begin{proof}
	By Theorem \ref{factor},
	there exists an $S\subseteq V(G)$ such that $c_{tc}(G-S)\geq |S|+1$. Let $a$, $b$ and $c$ denote the number of $K_1$, $K_3$ and triangular-cactus with at least $5$ vertices in $G-S$, respectively, and let $T_1, T_2, \ldots,T_c$ be the components of triangular-cactus with at least $5$ vertices. So we rewrite
	$c_{tc}(G-S)=a+b+c\geq |S|+1$, and then $|S|\leq a+b+c-1$.
	We take two vertices from every $K_3$ to form a set $R$ in $G-S$. Clearly, $|R|=2b$.
	
	Suppose that $a+b+c=1$, and then we have $|S|=0$. Since $n\geq5$, we have $a=b=0$ and $c=1$, which implies that $G$ is a triangular-cactus. If $n=5$, then there are two vertices $u$ and $v$ of degree $2$ from distinct two blocks and $|N_G(u)\cup N_G(v)|=3$. Take $X=N_G(u)\cup N_G(v)$. Then
	\[I'(G)\leq \frac{|X|}{iso(G-X)-1}=\frac{3}{2-1}= 3.\]
   If $n\geq 6$, then there exist three vertices $u,v$ and $w$ of degree $2$ from distinct three blocks. Take $X=N_G(u)\cup N_G(v)\cup N_G(w)$. Then $|X|\leq 6$ and hence
	\[I'(G)\leq \frac{|X|}{iso(G-X)-1}\leq\frac{6}{3-1}= 3.\]
	
	Suppose that $a+b+c\geq 2$. Note that there exist two vertices $u_i, v_i$ of degree $2$ from distinct blocks in $T_i$ for $i=1,2,\ldots,c$. Take $X_i=N(u_i)\cup N(v_i)$ and $X=\bigcup\limits_{i=1}^cX_i$. Clearly $|X|\leq 4c$ and $iso(G-X)\geq 2c$. Let $Y=S\cup R\cup X$. Then
	\[|Y|=|S|+|R|+|X|\leq a+b+c-1+2b+4c=a+3b+5c-1\]
	and $iso(G-Y)\geq a+b+2c\geq 2$.
	So we have
	\begin{align}
	I'(G)&\leq \frac{|Y|}{iso(G-Y)-1}\notag\\
	&\leq \frac{a+3b+5c-1}{a+b+2c-1}\notag\\
	&=3+\frac{-2a-c+2}{a+b+2c-1}\notag\\
	&\leq 3+\frac{-2a-c+2}{c+1}\notag\\
	&=2+\frac{-2a+3}{c+1}\\
	&\leq 2+\frac{3}{c+1}\\
	&\leq 5\notag.
	\end{align}
Thus we complete the proof of {\bf(i)}.
	
If $a\geq 1$, then, by inequality (1), $I'(G)\leq 3$. If $a=0$ and $c\geq 1$, then, by inequality (2), $I'(G)\leq \frac{7}{2}$. It remains to consider the case $a=c=0$ and $b\geq 2$ to complete the proofs of {\bf(ii)}, {\bf(iii)} and {\bf(iv)}. It implies that $|S|\leq b-1$.

Suppose that $b=2$, so we have $|S|\leq 1$. If $|S|=0$, then $n=6$. We have
\[I'(G)\leq \frac{|R|}{iso(G-R)-1}=\frac{4}{2-1}=4.\]
If $|S|=1$, then $n=7$. Let $S=\{s\}$. Since $G$ is not one of three graphs in Fig 1, we can take two vertices $u$ and $v$ from two distinct $K_3$'s
such that they are not adjacent to $s$. Set $X=V(G)\setminus \{s,u,v\}$. Then
\[I'(G)\leq \frac{|X|}{i(G-X)-1}=\frac{4}{3-1}=2.\]

Suppose that $b\geq 3$. Take $Y=S\cup R$. Then $|Y|=|S|+|R|$ and $iso(G-Y)= b$. If $|S|\leq b-2$, then
\[I'(G)\leq \frac{|Y|}{iso(G-Y)-1}\leq\frac{3b-2}{b-1}=3+\frac{1}{b-1}\leq \frac{7}{2}.\]
If $|S|= b-1$, then
\[I'(G)\leq \frac{|Y|}{iso(G-Y)-1}=\frac{3b-1}{b-1}=3+\frac{2}{b-1}.\]
Clearly, $I'(G)\leq 4$ when $b=3$ (i.e., $n=11$), $I'(G)\leq \frac{11}{3}$ when $b=4$ (i.e., $n=15$), and $I'(G)\leq \frac{7}{2}$ when $b\geq 5$. Therefore {\bf(ii)}, {\bf(iii)} and {\bf(iv)} hold.
\end{proof}

{\bf Remark 3. }
Similar to Remark 2, we can obtain that $I'(H_m)=\frac{3m-1}{m-1}$.
Therefore, we have $I'(H_2)=5$, $I'(H_3)=4$, $I'(H_4)=\frac{11}{3}$ and $I'(H_5)=\frac{7}{2}$. This implies that the bounds are tight in Lemma \ref{I'}.

Theorem \ref{th} is an immediate result of Lemma \ref{I'} {\bf(iv)}. We consider the binding number of a graph without $\{K_2,C_{2i+1}|i\geq 2\}$-factor in the following.

\begin{lemma}\label{bind}
	Let $G$ be a graph of order at least $5$. If $G$ admits no $\{K_2,C_{2i+1}|i\geq 2\}$-factors, then $bind(G)\leq \frac{4}{3}$.
\end{lemma}
\begin{proof}
	By Theorem \ref{factor},
	there exists an $S\subseteq V(G)$ such that $c_{tc}(G-S)\geq |S|+1$. Let $A$, $B$ and $C$ denote the collections of $K_1$, $K_3$ and triangular-cactus with at least $5$ vertices in $G-S$, respectively, and let
	$|A|=a$, $|B|=b$, $|C|=c$. Let $T_1, T_2, \ldots,T_c$ be the components of triangular-cactus with at least $5$ vertices. So we rewrite
	$c_{tc}(G-S)=a+b+c\geq |S|+1\geq 1$, and then $|S|\leq a+b+c-1$. Note that $|V(G)|\geq a+3b+5c+|S|$. If $a\neq 0$, we take $Y=V(G)\setminus S$. Note that $N_G(Y)=V(G)\setminus A\neq V(G)$. Thus
\begin{align*}
	bind(G)&\leq \frac{|N_G(Y)|}{|Y|}\\
	&=\frac{|V(G)|-a}{|V(G)|-|S|}\\
	&=1+\frac{|S|-a}{|V(G)|-|S|}\\
	&\leq 1+\frac{b+c-1}{a+3b+5c}\\
	&< 1+\frac{b+c}{3b+3c}\\
	&=\frac{4}{3}.
\end{align*}
We now divide two cases to complete the proof for $a= 0$ in the following.
	
	{\bf Case 1.} $c=0$.
	
	Note that $b\geq 2$ as $|V(G)|\geq 5$. Clearly $|S|\leq b-1$.
	Let $Y$ be a set of all vertices from $b-1$ $K_3$'s in $B$. Clearly $|Y|=3(b-1)$. Note that $N_G(Y)\subseteq Y\cup S\neq V(G)$,
	and then $|N_G(Y)|\leq 3(b-1)+|S|\leq 4(b-1)$.
	Thus
	\[
	bind(G)\leq \frac{|N_G(Y)|}{|Y|}\leq\frac{4(b-1)}{3(b-1)}=\frac{4}{3}.
	\]

	{\bf Case 2.} $c\geq 1$.

 Let $v$ be a vertex of degree $2$ in $T_1$. Then $|N_{T_1}(v)|=2$. Set $Y=V(G)\setminus (S\cup N_{T_1}(v))$. Then
$N_G(Y)\subseteq V(G)\setminus \{v\}.$
	Clearly
	$|Y|=|V(G)|-|S|-2$
	 and $|N_G(Y)|\leq |V(G)|-1$.
Then
\begin{align*}
	bind(G)&\leq \frac{|N_G(Y)|}{|Y|}\\
	&\leq\frac{|V(G)|-1}{|V(G)|-|S|-2}\\
	&=1+\frac{|S|+1}{|V(G)|-|S|-2}\\
	&\leq 1+\frac{|S|+1}{|V(B)|+|V(C)|-2}\\
	&\leq 1+\frac{b+c}{3b+5c-2}\\
	&\leq 1+\frac{b+c}{3b+3c}\\
	&=\frac{4}{3}.
	\end{align*}
The proof is complete.
\end{proof}

{\bf Remark 4. }
Recall the graph $H_m$ defined in Remark 2. Let $Y\subseteq V(H_m)$ such that $N_{H_m}(Y)\neq V(H_m)$. Then $|Y\cap V(K)|\leq 1$. If $|Y\cap V(K)|=1$, then
$|Y\cap (V(H_m)\setminus V(K))|=0$ as $N_{H_m}(Y) \neq V(H_m)$. Therefore we have $\frac{|N_G(Y)|}{|Y|}=\frac{|V(G)|-1}{1}>\frac{4}{3}$.
If $|Y\cap V(K)|=0$, let $m_i$ be the number of $K_3$'s such that $|Y\cap V(K_3)|=i$ in $H_m-V(K)$, where $i=1,2,3$. Clearly, $|Y|=m_1+2m_2+3m_3$ and $|N_{H_m}(Y)|= m-1+2m_1+3m_2+3m_3$. Since $N_{H_m}(Y) \neq V(H_m)$, we have $m_2+m_3\leq m-1$. Then
\begin{align*}
	\frac{|N_{H_m}(Y)|}{|Y|}&=\frac{m-1+2m_1+3m_2+3m_3}{m_1+2m_2+3m_3}  \\
	&=2+\frac{m-1-m_2-3m_3}{m_1+2m_2+3m_3}\\
	&\geq 2-\frac{2m_3}{m_1+2m_2+3m_3}\\
	&=\frac{4}{3}+\frac{2m_1+4m_2}{3(m_1+2m_2+3m_3)}\\
	&\geq\frac{4}{3}.
\end{align*}
By Lemma \ref{bind}, we have $bind(H_m)=\frac{4}{3}$, which implies that the bound is tight. In addition, we can also take a simple example to show the tightness of our bound. Let $G$ denote a triangular-cactus with $5$ vertices. Obviously, $G$ has no $\{K_2, C_{2i+1}|i\geq 2\}$-factor and $bind(G)=\frac{4}{3}$.

\end{document}